\documentclass[reqno]{amsart}
\usepackage{amssymb,amsmath}%
\usepackage{cite}
\usepackage{ifpdf}
\usepackage{enumitem}
\ifpdf
\usepackage[hyperindex]{hyperref}
\else
\expandafter \ifx \csname dvipdfm \endcsname \relax
\usepackage[hypertex,hyperindex]{hyperref}
\else
\usepackage[dvipdfm,hyperindex]{hyperref}
\fi
\fi

\newcommand{\mymod}[3]{#1 \equiv #2 \kern -0.5em \pmod{#3}}
\newcommand{\mynotmod}[3]{#1 \not \equiv #2 \kern -0.6em \pmod{#3}}

\numberwithin{equation}{section}
\newtheorem{theorem}{Theorem}[section]

\begin{document}
\title[A note on dual third order Jacobsthal vectors]{A note on dual third order Jacobsthal vectors}

\author[G. Cerda-Morales]{Gamaliel Cerda-Morales}  

\address{Gamaliel Cerda-Morales \newline
 Instituto de Matem\'aticas, Pontificia Universidad Cat\'olica de Valpara\'iso, Blanco Viel 596, Valpara\'iso, Chile.}
\email{gamaliel.cerda.m@mail.pucv.cl}



\begin{abstract}
Dual third order Jacobsthal and dual third order Jacobsthal-Lucas numbers are defined. In this study, we work on these dual numbers and we obtain the properties e.g. some quadratic identities, summation, norm, negadual third order Jacobsthal identities, Binet formulas and relations of them. We also define new vectors which are called dual third order Jacobsthal vectors and dual third order Jacobsthal-Lucas vectors. We give properties of these vectors to exert in geometry of dual space.
\end{abstract}

\maketitle

\section{Introduction}
Dual numbers which have lots of applications to modeling plane joint, to screw systems and to mechanics, were first invented by W. K. Clifford in 1873. The dual numbers extend to the real numbers has the form $d=a +\varepsilon b$, where $\varepsilon$ is the dual unit and $\varepsilon^{2}=0$, $\varepsilon\neq0$. The set $\mathbb{D}=\mathbb{R}[\varepsilon]=\{a +\varepsilon b:\ a,b\in \mathbb{R}\}$ is called dual number system and forms two dimensional commutative associative algebra over the real numbers. The algebra of dual numbers is a ring with the following addition and multiplication operations
\begin{equation}
\begin{aligned}
(a_{1}+\varepsilon b_{1})+(a_{2}+\varepsilon b_{2})&=(a_{1}+a_{2})+\varepsilon (b_{1}+b_{2}),\\
(a_{1}+\varepsilon b_{1})\cdot(a_{2}+\varepsilon b_{2})&=a_{1}a_{2}+\varepsilon (a_{1}b_{2}+a_{2}b_{2}).
\end{aligned}\label{ec:1}
\end{equation}

The equality of two dual numbers $d_{1}=a_{1}+\varepsilon b_{1}$ and $d_{2}=a_{a}+\varepsilon b_{2}$ is defined as, $d_{1}=d_{2}$ if and only if $a_{1}=a_{2}$ and $b_{1}=b_{2}$. The division of two dual numbers provided $a_{2}\neq0$ is given by
\begin{equation}\label{ec:2}
\frac{d_{1}}{d_{2}}=\frac{a_{1}}{a_{2}}+\varepsilon \left( \frac{b_{1}a_{2}-a_{1}b_{2}}{a_{2}^{2}}\right).
\end{equation}
The conjugate of the dual number $d=a +\varepsilon b$ is $\overline{d}=a -\varepsilon b$.

Vectors are used to study the analytic geometry of space, where they give simple ways to describe lines, planes, surfaces and curves in space. In this work we will speak on vectors of dual space.

Now, the set $\mathbb{D}^{3}=\{\overrightarrow{a}+\varepsilon \overrightarrow{b}:\ \overrightarrow{a},\overrightarrow{b}\in \mathbb{R}^{3}\}$ is a module on the ring $\mathbb{D}$ which is called $\mathbb{D}$-Module and the members of $\mathbb{D}^{3}$ are called dual vectors consisting of two real vectors. Also a dual vector $\overrightarrow{d}=\overrightarrow{a}+\varepsilon \overrightarrow{b}$ has another expression of the form
\begin{equation}\label{ec:3}
\overrightarrow{d}=(a_{1}+\varepsilon b_{1},a_{2}+\varepsilon b_{2},a_{3}+\varepsilon b_{3})=(d_{1},d_{2},d_{3}),
\end{equation}
where $d_{1}$, $d_{2}$, $d_{3}$ are dual numbers and $\overrightarrow{a}=(a_{1},a_{2},a_{3})$, $\overrightarrow{b}=(b_{1},b_{2},b_{3})$.

The norm of the dual vector $\overrightarrow{d}$ is given by
\begin{equation}\label{ec:4}
\left\Vert\overrightarrow{d}\right\Vert=\left\Vert\overrightarrow{a}\right\Vert +\varepsilon \frac{\langle \overrightarrow{a},\overrightarrow{b} \rangle}{\left\Vert\overrightarrow{a}\right\Vert},
\end{equation}
where $\langle \overrightarrow{a},\overrightarrow{b} \rangle= a_{1}b_{1}+a_{2}b_{2}+a_{3}b_{3}$. Furthermore, $\overrightarrow{d}=\overrightarrow{a}+\varepsilon \overrightarrow{b}$ is dual unit vector (e.g. $\left \Vert \overrightarrow{d}\right \Vert=1$) if and only if $\left\Vert\overrightarrow{a}\right\Vert=1$ and $\langle \overrightarrow{a},\overrightarrow{b} \rangle=0$.

The dual unit vectors are related with oriented lines, found by E. Study, which is called Study mapping: The oriented lines in $\mathbb{R}^{3}$ are in one-to-one correspondence with the points of dual unit sphere in $\mathbb{D}^{3}$.

In the other hand, the Jacobsthal numbers have many interesting properties and applications in many fields of science (see, e.g., \cite{Ba}). The Jacobsthal numbers $J_{n}$ are defined by the recurrence relation
\begin{equation}\label{e1}
J_{0}=0,\ J_{1}=1,\ J_{n+1}=J_{n}+2J_{n-1},\ n\geq1.
\end{equation}
Another important sequence is the Jacobsthal-Lucas sequence. This sequence is defined by the recurrence relation $j_{0}=2,\ j_{1}=1,\ j_{n+1}=j_{n}+2j_{n-1},\ n\geq1$. (see, \cite{Hor3}).

In \cite{Cook-Bac} the Jacobsthal recurrence relation (\ref{e1}) is extended to higher order recurrence relations and the basic list of identities provided by A. F. Horadam \cite{Hor3} is expanded and extended to several identities for some of the higher order cases. In particular, third order Jacobsthal numbers, $\{J_{n}^{(3)}\}_{n\geq0}$, and third order Jacobsthal-Lucas numbers, $\{j_{n}^{(3)}\}_{n\geq0}$, are defined by
\begin{equation}\label{ec:5}
J_{n+3}^{(3)}=J_{n+2}^{(3)}+J_{n+1}^{(3)}+2J_{n}^{(3)},\ J_{0}^{(3)}=0,\ J_{1}^{(3)}=J_{2}^{(3)}=1,\ n\geq0,
\end{equation}
and 
\begin{equation}\label{ec:6}
j_{n+3}^{(3)}=j_{n+2}^{(3)}+j_{n+1}^{(3)}+2j_{n}^{(3)},\ j_{0}^{(3)}=2,\ j_{1}^{(3)}=1,\ j_{2}^{(3)}=5,\ n\geq0,
\end{equation}
respectively.

The following properties given for third order Jacobsthal numbers and third order Jacobsthal-Lucas numbers play important roles in this paper (see \cite{Cer,Cer1,Cook-Bac}). 
\begin{equation}\label{e4}
3J_{n}^{(3)}+j_{n}^{(3)}=2^{n+1},
\end{equation}
\begin{equation}\label{e5}
j_{n}^{(3)}-3J_{n}^{(3)}=2j_{n-3}^{(3)},
\end{equation}
\begin{equation}\label{ec5}
J_{n+2}^{(3)}-4J_{n}^{(3)}=\left\{ 
\begin{array}{ccc}
-2 & \textrm{if} & \mymod{n}{1}{3} \\ 
1 & \textrm{if} & \mynotmod{n}{1}{3}%
\end{array}%
\right. ,
\end{equation}
\begin{equation}\label{e6}
j_{n}^{(3)}-4J_{n}^{(3)}=\left\{ 
\begin{array}{ccc}
2 & \textrm{if} & \mymod{n}{0}{3} \\ 
-3 & \textrm{if} & \mymod{n}{1}{3}\\ 
1 & \textrm{if} & \mymod{n}{2}{3}%
\end{array}%
\right. ,
\end{equation}
\begin{equation}\label{e7}
j_{n+1}^{(3)}+j_{n}^{(3)}=3J_{n+2}^{(3)},
\end{equation}
\begin{equation}\label{e8}
j_{n}^{(3)}-J_{n+2}^{(3)}=\left\{ 
\begin{array}{ccc}
1 & \textrm{if} & \mymod{n}{0}{3} \\ 
-1 & \textrm{if} & \mymod{n}{1}{3} \\ 
0 & \textrm{if} & \mymod{n}{2}{3}%
\end{array}%
\right. ,
\end{equation}
\begin{equation}\label{e9}
\left( j_{n-3}^{(3)}\right) ^{2}+3J_{n}^{(3)}j_{n}^{(3)}=4^{n},
\end{equation}
\begin{equation}\label{e10}
\sum\limits_{k=0}^{n}J_{k}^{(3)}=\left\{ 
\begin{array}{ccc}
J_{n+1}^{(3)} & \textrm{if} & \mynotmod{n}{0}{3} \\ 
J_{n+1}^{(3)}-1 & \textrm{if} & \mymod{n}{0}{3}%
\end{array}%
\right. 
\end{equation}
and
\begin{equation}\label{e12}
\left( j_{n}^{(3)}\right) ^{2}-9\left( J_{n}^{(3)}\right)^{2}=2^{n+2}j_{n-3}^{(3)}.
\end{equation}

Using standard techniques for solving recurrence relations, the auxiliary equation, and its roots are given by 
$$x^{3}-x^{2}-x-2=0;\ x = 2,\ \textrm{and}\ x=\frac{-1\pm i\sqrt{3}}{2}.$$ 

Note that the latter two are the complex conjugate cube roots of unity. Call them $\omega_{1}$ and $\omega_{2}$, respectively. Thus the Binet formulas can be written as
\begin{equation}\label{b1}
J_{n}^{(3)}=\frac{2}{7}2^{n}-\frac{3+2i\sqrt{3}}{21}\omega_{1}^{n}-\frac{3-2i\sqrt{3}}{21}\omega_{2}^{n}
\end{equation}
and
\begin{equation}\label{b2}
j_{n}^{(3)}=\frac{8}{7}2^{n}+\frac{3+2i\sqrt{3}}{7}\omega_{1}^{n}+\frac{3-2i\sqrt{3}}{7}\omega_{2}^{n},
\end{equation}
respectively. 

A variety of new results on Fibonacci-like quaternion and octonion numbers can be found in several papers \cite{Cer,Cim1,Cim2,Hal1,Hal2,Hor1,Hor2,Iye,Ke-Ak,Szy-Wl}. The origin of the topic of number sequences in division algebra can be traced back to the works by Horadam in \cite{Hor1} and by Iyer in \cite{Iye}. Horadam \cite{Hor1} defined the quaternions with the classic Fibonacci and Lucas number components as
\[
QF_{n}=F_{n}+F_{n+1}\textbf{i}+F_{n+2}\textbf{j}+F_{n+3}\textbf{k}
\]
and
\[
QL_{n}=L_{n}+L_{n+1}\textbf{i}+L_{n+2}\textbf{j}+L_{n+3}\textbf{k},
\]
respectively, where $F_{n}$ and $L_{n}$ are the $n$-th classic Fibonacci and Lucas numbers, respectively, and the author studied the properties of these quaternions. Several interesting and useful extensions of many of the familiar quaternion
numbers (such as the Fibonacci and Lucas quaternions \cite{Aky,Hal1,Hor1}, Pell quaternion \cite{Ca,Cim1} and Jacobsthal quaternions \cite{Szy-Wl}) have been considered by several authors.

There has been an increasing interest on quaternions and octonions that play an important role in various areas such as computer sciences, physics, differential geometry, quantum physics, signal, color image processing and geostatics (for more, see \cite{Ad,Car,Go1,Go2,Ko1,Ko2}). For example, in \cite{Cer,Cer1} the author studied the third order Jacobsthal quaternions and give some interesting properties of this numbers.

In this paper, we give some properties and relations of dual third order Jacobsthal and dual third order Jacobsthal-Lucas numbers. We define dual third order Jacobsthal vectors and investigate geometric notions which are created by using dual third order Jacobsthal vectors. 

\section{Dual Third Order Jacobsthal and Dual Third Order Jacobsthal-Lucas Numbers}\label{sect:2}
In this section, we define new kinds of sequences of dual number called as dual third order Jacobsthal numbers and dual third order Jacobsthal-Lucas numbers. We study some properties of these numbers. We obtain various results for these classes of dual numbers included recurrence relations, summation formulas, Binet's formulas and generating functions.

In \cite{Cer}, the author introduced the so-called dual Fibonacci numbers, which are a new class of dual numbers. They are defined by
\begin{equation}\label{eq:1}
FD_{n}=F_{n}+\varepsilon F_{n+1},\ (n\geq 0)
\end{equation}
where $F_{n}$ is the $n$-th Fibonacci number, $\varepsilon^{2}=0$ and $\varepsilon\neq 0$.

We now consider the usual third order Jacobsthal and third order Jacobsthal-Lucas numbers, and based on the definition (\ref{eq:1}) we give definitions of new kinds of dual numbers, which we call the dual third order Jacobsthal numbers and dual third order Jacobsthal-Lucas numbers. In this paper, we define the $n$-th dual third order Jacobsthal number and dual third order Jacobsthal-Lucas number, respectively, by the following recurrence relations
\begin{equation}\label{eq:2}
JD_{n}^{(3)}=J_{n}^{(3)}+\varepsilon J_{n+1}^{(3)},\ n\geq 0
\end{equation}
and
\begin{equation}\label{eq:3}
jD_{n}^{(3)}=j_{n}^{(3)}+\varepsilon j_{n+1}^{(3)},\ n\geq 0,
\end{equation}
where $J_{n}^{(3)}$ and $j_{n}^{(3)}$ are the $n$-th third order Jacobsthal number and third order Jacobsthal-Lucas number, respectively. 

The equalities in (\ref{ec:1}) gives 
\begin{equation}\label{eq:4}
JD_{n}^{(3)}\pm jD_{n}^{(3)}=(J_{n}^{(3)}\pm j_{n}^{(3)})+\varepsilon (J_{n+1}^{(3)}\pm j_{n+1}^{(3)}).
\end{equation}
From the conjugate of a dual number, (\ref{eq:2}) and (\ref{eq:3}) an easy computation gives 
\begin{equation}\label{eq:5}
\overline{JD_{n}^{(3)}}=J_{n}^{(3)}-\varepsilon J_{n+1}^{(3)} ,\ \overline{jD_{n}^{(3)}}=j_{n}^{(3)}-\varepsilon j_{n+1}^{(3)}.
\end{equation}

By some elementary calculations we find the following recurrence relations for the dual third order Jacobsthal and dual third order Jacobsthal-Lucas numbers from (\ref{eq:2}), (\ref{eq:3}), (\ref{eq:4}), (\ref{ec:1}), (\ref{ec:5}) and (\ref{ec:6}):
\begin{equation}
\begin{aligned}
JD_{n+1}^{(3)}+JD_{n}^{(3)}+2JD_{n-1}^{(3)}&=(J_{n+1}^{(3)}+J_{n}^{(3)}+2J_{n-1}^{(3)})+\varepsilon (J_{n+2}^{(3)}+J_{n+1}^{(3)}+2J_{n}^{(3)})\\
&=J_{n+2}^{(3)}+\varepsilon J_{n+3}^{(3)}\\
&=JD_{n+2}^{(3)}\label{equ:3}
\end{aligned} 
\end{equation}
and similarly $jD_{n+2}^{(3)}=jD_{n+1}^{(3)}+jD_{n}^{(3)}+2jD_{n-1}^{(3)}$, for $n\geq1$.

Now, we will state Binet's formulas for the third order Jacobsthal and third order Jacobsthal-Lucas octonions. Repeated use of (\ref{b1}) in (\ref{eq:2}) enables one to write for $\underline{\alpha}=1+2\varepsilon $, $\underline{\omega_{1}}=1+\omega_{1}\varepsilon $ and $\underline{\omega_{2}}=1+\omega_{2}\varepsilon $ 
\begin{equation}
\begin{aligned}
JD_{n}^{(3)}&=J_{n}^{(3)}+\varepsilon J_{n+1}^{(3)}\\
&=\frac{2}{7}2^{n}-\frac{3+2i\sqrt{3}}{21}\omega_{1}^{n}-\frac{3-2i\sqrt{3}}{21}\omega_{2}^{n}\\
&\ \ +\varepsilon \left(\frac{2}{7}2^{n+1}-\frac{3+2i\sqrt{3}}{21}\omega_{1}^{n+1}-\frac{3-2i\sqrt{3}}{21}\omega_{2}^{n+1}\right)\\
&=\frac{2}{7}\underline{\alpha}2^{n}-\frac{3+2i\sqrt{3}}{21}\underline{\omega_{1}}\omega_{1}^{n}-\frac{3-2i\sqrt{3}}{21}\underline{\omega_{2}}\omega_{2}^{n}
\end{aligned} \label{equ:4}
\end{equation}
and similarly making use of (\ref{b2}) in (\ref{eq:3}) yields
\begin{equation}
\begin{aligned}
jD_{n}^{(3)}&=j_{n}^{(3)}+\varepsilon j_{n+1}^{(3)}\\
&=\frac{8}{7}2^{n}+\frac{3+2i\sqrt{3}}{7}\omega_{1}^{n}+\frac{3-2i\sqrt{3}}{7}\omega_{2}^{n}\\
&\ \ + \varepsilon \left(\frac{8}{7}2^{n+1}+\frac{3+2i\sqrt{3}}{7}\omega_{1}^{n+1}+\frac{3-2i\sqrt{3}}{7}\omega_{2}^{n+1}\right)\\
&=\frac{8}{7}\underline{\alpha}2^{n}+\frac{3+2i\sqrt{3}}{7}\underline{\omega_{1}}\omega_{1}^{n}+\frac{3-2i\sqrt{3}}{7}\underline{\omega_{2}}\omega_{2}^{n}.
\end{aligned} \label{equ:5}
\end{equation}
The formulas in (\ref{equ:4}) and (\ref{equ:5}) are called as Binet's formulas for the dual third order Jacobsthal and dual third order Jacobsthal-Lucas numbers, respectively. The recurrence relations for the $n$-th dual third order Jacobsthal number are expressed in the following theorem.

\begin{theorem}\label{th:1}
For $n, m\geq 0$, we have the following identities:
\begin{equation}\label{t1}
JD_{n+2}^{(3)}+JD_{n+1}^{(3)}+JD_{n}^{(3)}=2^{n+1}JD_{2}^{(3)},
\end{equation}
\begin{equation}\label{t2}
JD_{n+2}^{(3)}-4JD_{n}^{(3)}=\left\{ 
\begin{array}{ccc}
1-2\varepsilon  & \textrm{if} & \mymod{n}{0}{3} \\ 
-2+\varepsilon & \textrm{if} & \mymod{n}{1}{3}\\
1+\varepsilon & \textrm{if} & \mymod{n}{2}{3} 
\end{array}%
\right. ,
\end{equation}
\begin{equation}\label{t3}
JD_{n}^{(3)}JD_{m+1}^{(3)}+T_{n-1}^{(3)}JD_{m}^{(3)}+2JD_{n-1}^{(3)}JD_{m-1}^{(3)}=JD_{n+m}^{(3)}+\varepsilon J_{n+m+1}^{(3)},
\end{equation}
\begin{equation}\label{t4}
\left(JD_{n+1}^{(3)}\right)^{2}+\left(JD_{n}^{(3)}\right)^{2}+4JD_{n}^{(3)}JD_{n-1}^{(3)}=JD_{2n+1}^{(3)}+\varepsilon J_{2n+2}^{(3)},
\end{equation}
where $JD_{2}^{(3)}=1+2\varepsilon$ and $T_{n}^{(3)}=JD_{n}^{(3)}+2JD_{n-1}^{(3)}$.
\end{theorem}
\begin{proof}
Consider (\ref{eq:2}) and (\ref{eq:4}) we can write
\begin{equation}\label{p1}
JD_{n+2}^{(3)}+JD_{n+1}^{(3)}+JD_{n}^{(3)}=J_{n+2}^{(3)}+J_{n+1}^{(3)}+J_{n}^{(3)}+\varepsilon (J_{n+3}^{(3)}+J_{n+2}^{(3)}+J_{n+1}^{(3)}).
\end{equation}
Using the identity $J_{n+2}^{(3)}+J_{n+1}^{(3)}+J_{n}^{(3)}=2^{n+1}$, the above sum can be calculated as $$JD_{n+2}^{(3)}+JD_{n+1}^{(3)}+JD_{n}^{(3)}=2^{n+1}+2^{n+2}\varepsilon,$$ which can be simplified as $JD_{n+2}^{(3)}+JD_{n+1}^{(3)}+JD_{n}^{(3)}=2^{n+1}(1+2\varepsilon)$. Now, using (\ref{ec5}) and (\ref{eq:2}) we can write $JD_{n+2}^{(3)}-4JD_{n}^{(3)}=1-2\varepsilon$ if $\mymod{n}{1}{3}$ and similarly in the other cases, this proves (\ref{t2}). Now, from the definition of third order Jacobsthal number, dual third order Jacobsthal number in Eq. (\ref{eq:2}), the equations $$\left(J_{n+1}^{(3)}\right)^{2}+\left(J_{n}^{(3)}\right)^{2}+4J_{n}^{(3)}J_{n-1}^{(3)}=J_{2n+1}^{(3)}$$ and $J_{n}^{(3)}J_{m+1}^{(3)}+(J_{n-1}^{(3)}+2J_{n-2}^{(3)})J_{m}^{(3)}+2J_{n-1}^{(3)}J_{m-1}^{(3)}=J_{n+m}^{(3)}$ (see Waddill and Sacks \cite{Wa}), we get
\begin{equation}
\begin{aligned}
JD_{n}^{(3)}JD_{m+1}^{(3)}+&(JD_{n-1}^{(3)}+2JD_{n-2}^{(3)})JD_{m}^{(3)}+2JD_{n-1}^{(3)}JD_{m-1}^{(3)}\\
&=(J_{n}^{(3)}+\varepsilon J_{n+1}^{(3)})(J_{m+1}^{(3)}+\varepsilon J_{m+2}^{(3)})\\
&\ \ +((J_{n-1}^{(3)}+2J_{n-2}^{(3)})+\varepsilon(J_{n}^{(3)}+2J_{n-1}^{(3)}))(J_{m}^{(3)}+\varepsilon J_{m+1}^{(3)})\\
&\ \ +2(J_{n-1}^{(3)}+\varepsilon J_{n}^{(3)})(J_{m-1}^{(3)}+\varepsilon J_{m}^{(3)})\\
&=(J_{n}^{(3)}J_{m+1}^{(3)}+(J_{n-1}^{(3)}+2J_{n-2}^{(3)})J_{m}^{(3)}+2J_{n-1}^{(3)}J_{m-1}^{(3)})\\
&\ \ +\varepsilon (J_{n}^{(3)}J_{m+2}^{(3)}+(J_{n-1}^{(3)}+2J_{n-2}^{(3)})J_{m+1}^{(3)}+2J_{n-1}^{(3)}J_{m}^{(3)})\\
&\ \ +\varepsilon (J_{n+1}^{(3)}J_{m+1}^{(3)}+(J_{n}^{(3)}+2J_{n-1}^{(3)})J_{m}^{(3)}+2J_{n}^{(3)}J_{m-1}^{(3)})\\
&=(J_{n+m}^{(3)}+\varepsilon J_{n+m+1}^{(3)})+\varepsilon J_{n+m+1}^{(3)}\\
&=JD_{n+m}^{(3)}+\varepsilon J_{n+m+1}^{(3)}.
\end{aligned}\label{p2}
\end{equation}
Finally, if we consider first $n=n+1$ and $m=n$ in above result (\ref{p2}), we obtain
\[
\left(JD_{n+1}^{(3)}\right)^{2}+\left(JD_{n}^{(3)}\right)^{2}+4JD_{n}^{(3)}JD_{n-1}^{(3)}=JD_{2n+1}^{(3)}+\varepsilon J_{2n+2}^{(3)},
\]
which is the assertion (\ref{t4}) of theorem.
\end{proof}

The following theorem deals with two relations between the dual third order Jacobsthal and dual third order Jacobsthal-Lucas numbers.
\begin{theorem}\label{th:2}
Let $n\geq 0$ be integer. Then,
\begin{equation}\label{t7}
jD_{n+3}^{(3)}-3JD_{n+3}^{(3)}=2jD_{n}^{(3)},
\end{equation}
\begin{equation}\label{t8}
jD_{n}^{(3)}+jD_{n+1}^{(3)}=3JD_{n+2}^{(3)},
\end{equation}
\begin{equation}\label{t9}
jD_{n}^{(3)}-JD_{n+2}^{(3)}=\left\{ 
\begin{array}{ccc}
1-\varepsilon & \textrm{if} & \mymod{n}{0}{3} \\ 
-1& \textrm{if} & \mymod{n}{1}{3}\\
\varepsilon & \textrm{if} & \mymod{n}{2}{3} 
\end{array}%
\right. ,
\end{equation}
\begin{equation}\label{t10}
jD_{n}^{(3)}-4JD_{n}^{(3)}=\left\{ 
\begin{array}{ccc}
2-3\varepsilon & \textrm{if} & \mymod{n}{0}{3} \\ 
-3+\varepsilon & \textrm{if} & \mymod{n}{1}{3}\\
1+2\varepsilon & \textrm{if} & \mymod{n}{2}{3} 
\end{array}%
\right. .
\end{equation}
\end{theorem}
\begin{proof}
The following recurrence relation 
\begin{equation}\label{p3}
jD_{n+3}^{(3)}-3JD_{n+3}^{(3)}=(j_{n+3}^{(3)}-3J_{n+3}^{(3)})+\varepsilon (j_{n+4}^{(3)}-3J_{n+4}^{(3)})
\end{equation}
can be readily written considering that $JD_{n}^{(3)}=J_{n}^{(3)}+\varepsilon j_{n+1}^{(3)}$ and $j_{n}^{(3)}=j_{n}^{(3)}+\varepsilon j_{n+1}^{(3)}$. Notice that $j_{n+3}^{(3)}-3J_{n+3}^{(3)}=2j_{n}^{(3)}$ from (\ref{e5}) (see \cite{Cook-Bac}), whence it follows that (\ref{p3}) can be rewritten as $jD_{n+3}^{(3)}-3JD_{n+3}^{(3)}=2jD_{n}^{(3)}$ from which the desired result (\ref{t7}) of Theorem \ref{th:2}. In a similar way we can show the second equality. By using the identity $j_{n}^{(3)}+j_{n+1}^{(3)}=3J_{n+2}^{(3)}$ we have $$jD_{n}^{(3)}+jD_{n+1}^{(3)}=3(J_{n+2}^{(3)}+\varepsilon J_{n+3}^{(3)}),$$ which is the assertion (\ref{t8}) of theorem.

By using the identity $j_{n}^{(3)}-J_{n+2}^{(3)}=1$ from (\ref{e8}) (see \cite{Cook-Bac}) we have
$$jD_{n}^{(3)}-JD_{n+2}^{(3)}=(j_{n}^{(3)}-J_{n+2}^{(3)})+\varepsilon (j_{n+1}^{(3)}-J_{n+3}^{(3)})=1-\varepsilon
$$
if $\mymod{n}{0}{3}$, the other identities are clear from equation (\ref{e8}). Finally, the proof of Eq. (\ref{t10}) is similar to (\ref{t9}) by using (\ref{e6}).
\end{proof}

Now, we use the notation
\begin{equation}\label{h5}
H_{n}(a,b)=\frac{A\omega_{1}^{n}-B\omega_{2}^{n}}{\omega_{1}-\omega_{2}}=\left\{ 
\begin{array}{ccc}
a & \textrm{if} & \mymod{n}{0}{3} \\ 
b& \textrm{if} & \mymod{n}{1}{3} \\ 
-(a+b)& \textrm{if} & \mymod{n}{2}{3}%
\end{array}%
\right. ,
\end{equation}
where $A=b-a\omega_{2}$ and $B=b-a\omega_{1}$, in which $\omega_{1}$ and $\omega_{2}$ are the complex conjugate cube roots of unity (i.e. $\omega_{1}^{3}=\omega_{2}^{3}=1$). Furthermore, note that for all $n\geq0$ we have 
\begin{equation}
H_{n+2}(a,b)=-H_{n+1}(a,b)-H_{n}(a,b),
\end{equation}
where $H_{0}(a,b)=a$ and $H_{1}(a,b)=b$.

From the Binet formulas (\ref{b1}), (\ref{b2}) and Eq. (\ref{h5}), we have
\begin{equation}
J_{n}^{(3)}=\frac{1}{7}\left(2^{n+1}-V_{n}\right)\ \textrm{and}\ j_{n}^{(3)}=\frac{1}{7}\left(2^{n+3}+3V_{n}\right),
\end{equation}
where $V_{n}=H_{n}(2,-3)$. Then, for $m\geq n$:
\begin{equation}\label{p4}
\begin{aligned}
J_{m}^{(3)}J_{n+1}^{(3)}-J_{m+1}^{(3)}J_{n}^{(3)}&=\frac{1}{49}\left(\begin{array}{ccc}(2^{m+1}-V_{m})(2^{n+2}-V_{n+1})\\
-(2^{m+2}-V_{m+1})(2^{n+1}-V_{n})\end{array}%
\right)\\
&=\frac{1}{49}\left( 
\begin{array}{ccc}
-2^{m+1}V_{n+1}-2^{n+2}V_{m}+2^{m+2}V_{n}+2^{n+1}V_{m+1}\\
+V_{m}V_{n+1}-V_{m+1}V_{n}
\end{array}%
\right)\\
&=\frac{1}{7}\left( 
\begin{array}{ccc} 2^{m+1}U_{n+1}-2^{n+1}U_{m+1}+U_{m-n}
\end{array}%
\right),
\end{aligned}
\end{equation}
where $U_{n+1}=\frac{1}{7}(2V_{n}-V_{n+1})=H_{n+1}(0,1)$ and $V_{n}=H_{n}(2,-3)$. Furthermore, if $m=n+1$ in Eq. (\ref{p4}), we obtain for $n\geq 0$,
\begin{equation}\label{p5}
J_{n+2}^{(3)}J_{n}^{(3)}-\left(J_{n+1}^{(3)}\right)^{2}=\frac{1}{7}\left(2^{n+1}V_{-(n+2)}-1\right),
\end{equation}
where $V_{-n}=U_{n}-2U_{n+2}=H_{n}(2,1)$. Using the above notation, the following theorem investigate a type of Cassini identity for this numbers. 
\begin{theorem}\label{thm:3}
For $n\geq 0$, the Cassini-like identity for dual third order Jacobsthal number $JD_{n}^{(3)}$ is given by
\begin{equation}\label{p6}
JD_{n+2}^{(3)}JD_{n}^{(3)}-\left(JD_{n+1}^{(3)}\right)^{2}=\frac{1}{7}\left(2^{n+1}VD_{-(n+2)}+(-1+\varepsilon(2^{n+2}V_{-(n+2)}+1))\right),
\end{equation}
where $V_{-n}=H_{n}(2,1)$ and $VD_{-n}=V_{-n}+\varepsilon V_{-(n+1)}$.
\end{theorem}
\begin{proof}
From Eqs. (\ref{eq:2}) and (\ref{eq:4}), the identity (\ref{p5}) for third order Jacobsthal numbers and $n=m+2$ in Eq. (\ref{p4}), we get
\begin{align*}
JD_{n+2}^{(3)}JD_{n}^{(3)}-\left(JD_{n+1}^{(3)}\right)^{2}&=\left(J_{n+2}^{(3)}+\varepsilon J_{n+3}^{(3)}\right)\left(J_{n}^{(3)}+\varepsilon J_{n+1}^{(3)}\right)-\left(J_{n+1}^{(3)}+\varepsilon J_{n+2}^{(3)}\right)^{2}\\
&=\left(J_{n+2}^{(3)}J_{n}^{(3)}-\left(J_{n+1}^{(3)}\right)^{2}\right)+\varepsilon \left(J_{n+3}^{(3)}J_{n}^{(3)}-J_{n+1}^{(3)}J_{n+2}^{(3)}\right)\\
&=\frac{1}{7}\left(2^{n+1}V_{-(n+2)}-1\right)+\frac{\varepsilon}{7}\left(2^{n+1}(V_{-n}+2V_{-(n+2)})+1\right),
\end{align*}
where $U_{n}-4U_{n+1}=V_{-n}+2V_{-(n+2)}=H_{n}(-4,5)$. 

Furthermore, using $VD_{-(n+2)}=V_{-(n+2)}+\varepsilon V_{-n}$, we obtain the next result
\begin{align*}
JD_{n+2}^{(3)}JD_{n}^{(3)}-\left(JD_{n+1}^{(3)}\right)^{2}&=\frac{1}{7}\left(2^{n+1}V_{-(n+2)}-1 + 2^{n+1}\varepsilon(V_{-n}+2V_{-(n+2)})+\varepsilon\right)\\
&=\frac{1}{7}\left(2^{n+1}VD_{-(n+2)}+(-1+\varepsilon(2^{n+2}V_{-(n+2)}+1))\right).
\end{align*}
we reach (\ref{p6}).
\end{proof}

\begin{theorem}\label{thm:3}
If $JD_{n}^{(3)}$ is a dual third order Jacobsthal number, then the limit of consecutive quotients of this numbers is 
\begin{equation}\label{p7}
\lim_{n\rightarrow \infty}\frac{JD_{n+1}^{(3)}}{JD_{n}^{(3)}}=\lim_{n\rightarrow \infty}\left(\frac{J_{n+1}^{(3)}+\varepsilon J_{n+2}^{(3)}}{J_{n}^{(3)}+\varepsilon J_{n+1}^{(3)}}\right)=2.
\end{equation}
\end{theorem}
\begin{proof}
The limit of consecutive quotients of third order Jacosbthal numbers approaches to the radio $\frac{J_{n+1}^{(3)}}{J_{n}^{(3)}}\rightarrow 2$ if $n\rightarrow \infty$ (See \cite{Cook-Bac}). From the previous limit, Eqs. (\ref{eq:2}) and (\ref{p5}), we have
\begin{equation}\label{p8}
\begin{aligned}
\lim_{n\rightarrow \infty}\frac{J_{n+1}^{(3)}+\varepsilon J_{n+2}^{(3)}}{J_{n}^{(3)}+\varepsilon J_{n+1}^{(3)}}&=\lim_{n\rightarrow \infty}\left(\frac{J_{n}^{(3)}J_{n+1}^{(3)}+\varepsilon \left(J_{n+2}^{(3)}J_{n}^{(3)}-\left(J_{n+1}^{(3)}\right)^{2}\right)}{\left(J_{n}^{(3)}\right)^{2}}\right)\\
&=\lim_{n\rightarrow \infty}\frac{J_{n+1}^{(3)}}{J_{n}^{(3)}}+\varepsilon \lim_{n\rightarrow \infty}\left(\frac{2^{n+1}V_{-(n+2)}-1}{7\left(J_{n}^{(3)}\right)^{2}}\right),
\end{aligned}
\end{equation}
where $V_{-n}=H_{n}(2,1)$. In last equality of (\ref{p8}), by using $V_{-n}=H_{n}(2,1)$ (see Eq. (\ref{h5})), $\lim_{n\rightarrow \infty}\frac{2^{n+1}}{7J_{n}^{(3)}}=1$ and $$\lim_{n\rightarrow \infty}\left(\frac{J_{n+2}^{(3)}-2J_{n+1}^{(3)}}{J_{n}^{(3)}}\right)=\lim_{n\rightarrow \infty}\left(\frac{J_{n+2}^{(3)}}{J_{n}^{(3)}}-4\frac{J_{n+1}^{(3)}}{J_{n}^{(3)}}\right)=0,$$ we find zero for the second limit. Thus, the result (\ref{p7}) is true.
\end{proof}

\section{Dual Third Order Jacobsthal Vectors}
A dual vector in $\mathbb{D}^{3}$ is given by $\overrightarrow{d}=\overrightarrow{a}+\varepsilon \overrightarrow{b}$ where $\overrightarrow{a}, \overrightarrow{b}\in \mathbb{R}^{3}$. Now, we will give dual third order Jacobsthal vectors and geometric properties of them. 

A dual third order Jacobsthal vector is defined by
\begin{equation}\label{p9}
\overrightarrow{JD_{n}^{(3)}}=\overrightarrow{J_{n}^{(3)}}+\varepsilon \overrightarrow{J_{n+1}^{(3)}},\ n\geq 0,
\end{equation}
where $\overrightarrow{J_{n}^{(3)}}=\left(J_{n}^{(3)},J_{n+1}^{(3)},J_{n+2}^{(3)}\right)$ and $\overrightarrow{J_{n+1}^{(3)}}=\left(J_{n+1}^{(3)},J_{n+2}^{(3)},J_{n+3}^{(3)}\right)$ are real vectors in $\mathbb{R}^{3}$ with $n$-th third order Jacobsthal number $J_{n}^{(3)}$.

The dual third order Jacobsthal vector $\overrightarrow{JD_{n}^{(3)}}$ is also can be expressed as 
\begin{equation}
\overrightarrow{JD_{n}^{(3)}}=\left(JD_{n}^{(3)},JD_{n+1}^{(3)},JD_{n+2}^{(3)}\right),
\end{equation}
where $JD_{n}^{(3)}$ is the $n$-th dual third order Jacobsthal number. For example, the first three dual third order Jacobsthal vectors can be given easily as 
\begin{equation}
\begin{aligned}
\overrightarrow{JD_{0}^{(3)}}=\left(\varepsilon,1+\varepsilon,1+2\varepsilon\right),\\
\overrightarrow{JD_{1}^{(3)}}=\left(1+\varepsilon,1+2\varepsilon,2+5\varepsilon\right),\\
\overrightarrow{JD_{2}^{(3)}}=\left(1+2\varepsilon,2+5\varepsilon,5+9\varepsilon\right).
\end{aligned}\label{p10}
\end{equation}

Let $\overrightarrow{JD_{n}^{(3)}}$ and $\overrightarrow{JD_{m}^{(3)}}$ be two dual third order Jacobsthal vectors and $\lambda\in \mathbb{R}[\varepsilon]$ be a dual number. Then the product of the dual third order Jacobsthal vector and the scalar $\lambda$ is given by
\begin{equation}\label{p11}
\lambda\cdot \overrightarrow{JD_{n}^{(3)}}=\lambda \overrightarrow{J_{n}^{(3)}}+\varepsilon \lambda \overrightarrow{J_{n+1}^{(3)}}.
\end{equation}
Furthermore, $\overrightarrow{JD_{n}^{(3)}}=\overrightarrow{JD_{m}^{(3)}}$ if and only if $JD_{n}^{(3)}=JD_{m}^{(3)}$, $JD_{n+1}^{(3)}=JD_{m+1}^{(3)}$ and $JD_{n+2}^{(3)}=JD_{m+2}^{(3)}$, where $JD_{n}^{(3)}=J_{n}^{(3)}+\varepsilon J_{n+1}^{(3)}$.

\begin{theorem}\label{th:4}
The dual third order Jacobsthal vector $\overrightarrow{JD_{n}^{(3)}}$ is a dual unit vector if and only if 
\begin{equation}\label{pp12}
3\cdot 2^{2(n+1)}-2^{n+2}U_{n}=5\ \textrm{and}\ 3\cdot 2^{2n+3}-2^{n+1}(U_{n}-U_{n+2})=1,
\end{equation}
where $U_{n}=H_{n}(0,1)$.
\end{theorem}
\begin{proof}
By using the definitions of third order Jacobsthal numbers, Eq. (\ref{p9}) and the identities $V_{n}V_{n+1}+V_{n+1}V_{n+2}+V_{n+2}V_{n+3}=-7$ and $V_{n}^{2}+V_{n+1}^{2}+V_{n+2}^{2}=14$ (see Eq. (\ref{h5})) we get the following statements
\begin{align*}
\left\Vert \overrightarrow{J_{n}^{(3)}} \right\Vert ^{2}&=\left(J_{n}^{(3)}\right)^{2}+\left(J_{n+1}^{(3)}\right)^{2}+\left(J_{n+2}^{(3)}\right)^{2}\\
&=\frac{1}{49}\left( \left(2^{n+1}-V_{n}\right)^{2}+\left(2^{n+2}-V_{n+1}\right)^{2}+\left(2^{n+3}-V_{n+2}\right)^{2}\right)\\
&=\frac{1}{49}\left( 21\cdot 2^{2(n+1)}-2^{n+2}(V_{n}+2V_{n+1}+4V_{n+2})+14\right)\\
&=\frac{1}{7}\left( 3\cdot 2^{2(n+1)}-2^{n+2}U_{n}+2\right)
\end{align*}
and
\begin{align*}
\left\langle \overrightarrow{J_{n}^{(3)}},\overrightarrow{J_{n+1}^{(3)}} \right\rangle &=J_{n}^{(3)}J_{n+1}^{(3)}+J_{n+1}^{(3)}J_{n+2}^{(3)}+J_{n+2}^{(3)}J_{n+3}^{(3)}\\
&=\frac{1}{49}\left(\begin{array}{c} 
(2^{n+1}-V_{n})(2^{n+2}-V_{n+1})+(2^{n+2}-V_{n+1})(2^{n+3}-V_{n+2})\\
+(2^{n+3}-V_{n+2})(2^{n+4}-V_{n+3})
\end{array}\right)\\
&=\frac{1}{49}\left(\begin{array}{c} 
21\cdot 2^{2n+3}-2^{n+1}(4V_{n+3}+10V_{n+2}+5V_{n+1}+2V_{n})\\
+V_{n}V_{n+1}+V_{n+1}V_{n+2}+V_{n+2}V_{n+3}
\end{array}\right)\\
&=\frac{1}{7}\left(3\cdot 2^{2n+3}-2^{n+1}(U_{n}-U_{n+2})-1\right),
\end{align*}
where $7U_{n}=3V_{n+2}+V_{n+1}$, $V_{n}+5V_{n+2}=U_{n}-U_{n+2}$ and $U_{n}=H_{n}(0,1)$. 

Using that $\left \Vert \overrightarrow{JD_{n}^{(3)}}\right \Vert=1$ if and only if $\left\Vert\overrightarrow{J_{n}^{(3)}}\right\Vert=1$ and $\left\langle \overrightarrow{J_{n}^{(3)}},\overrightarrow{J_{n+1}^{(3)}} \right\rangle=0$ (see Eq. (\ref{ec:4})) and above calculations, we easily reach the result (\ref{pp12}).
\end{proof}

Now, if $\overrightarrow{d_{1}}=\overrightarrow{a_{1}}+\varepsilon \overrightarrow{b_{1}}$ and $\overrightarrow{d_{2}}=\overrightarrow{a_{2}}+\varepsilon \overrightarrow{b_{2}}$ are two dual vectors, then the dot product and cross product of them are given respectively by
\begin{equation}
\begin{aligned}
\left\langle \overrightarrow{d_{1}},\overrightarrow{d_{2}} \right\rangle&=\left\langle \overrightarrow{a_{1}},\overrightarrow{a_{2}} \right\rangle+\varepsilon \left(\left\langle \overrightarrow{a_{1}},\overrightarrow{b_{2}} \right\rangle +\left\langle \overrightarrow{b_{1}},\overrightarrow{a_{2}} \right\rangle \right),\\
\overrightarrow{d_{1}} \times \overrightarrow{d_{2}} &=\overrightarrow{a_{1}}\times \overrightarrow{a_{2}} +\varepsilon \left(\overrightarrow{a_{1}}\times \overrightarrow{b_{2}} +\overrightarrow{b_{1}}\times \overrightarrow{a_{2}} \right).
\end{aligned}\label{ec:Gu}
\end{equation}
(For more details, see \cite{Gu}).

\begin{theorem}\label{th:5}
Let $\overrightarrow{JD_{n}^{(3)}}$ and $\overrightarrow{JD_{m}^{(3)}}$ be two dual third order Jacobsthal vectors. The dot product of these two vectors is given by 
\begin{equation}\label{p14}
\begin{aligned}
\left\langle \overrightarrow{JD_{n}^{(3)}},\overrightarrow{JD_{m}^{(3)}} \right\rangle&=\frac{1}{7} \left(\begin{array}{c} 
3\cdot2^{n+m+2}(1+4\varepsilon)-2^{n+1}(UD_{m}+2\varepsilon U_{m})\\
-2^{m+1}(UD_{n}+\varepsilon U_{n})+W_{n-m}(1-\varepsilon)
\end{array}\right),
\end{aligned}
\end{equation}
where $U_{n}=H_{n}(0,1)$, $W_{n}=H_{n}(2,-1)$ and $UD_{n}=U_{n}+\varepsilon U_{n+1}$.
\end{theorem}
\begin{proof}
If $\overrightarrow{JD_{n}^{(3)}}=\overrightarrow{J_{n}^{(3)}}+\varepsilon \overrightarrow{J_{n+1}^{(3)}}$ and $\overrightarrow{JD_{m}^{(3)}}=\overrightarrow{J_{m}^{(3)}}+\varepsilon \overrightarrow{J_{m+1}^{(3)}}$ are two dual vectors,
then the dot product of them are given respectively by 
\begin{align*}
\left\langle \overrightarrow{JD_{n}^{(3)}},\overrightarrow{JD_{m}^{(3)}} \right\rangle&=\left\langle \overrightarrow{J_{n}^{(3)}},\overrightarrow{J_{m}^{(3)}} \right\rangle +\varepsilon \left( \left\langle \overrightarrow{J_{n}^{(3)}},\overrightarrow{J_{m+1}^{(3)}} \right\rangle +\left\langle \overrightarrow{J_{n+1}^{(3)}},\overrightarrow{J_{m}^{(3)}} \right\rangle \right)\\
&=J_{n}^{(3)}J_{m}^{(3)}+J_{n+1}^{(3)}J_{m+1}^{(3)}+J_{n+2}^{(3)}J_{m+2}^{(3)}\\
&\ \ +\varepsilon \left(\begin{array}{c} 
J_{n}^{(3)}J_{m+1}^{(3)}+J_{n+1}^{(3)}J_{m+2}^{(3)}+J_{n+2}^{(3)}J_{m+3}^{(3)}\\
+J_{n+1}^{(3)}J_{m}^{(3)}+J_{n+2}^{(3)}J_{m+1}^{(3)}+J_{n+3}^{(3)}J_{m+2}^{(3)}
\end{array}\right).
\end{align*}
By using the definition of third order Jacobsthal number (\ref{b1}), the equations (\ref{h5}) and (\ref{p9}), we have 
\begin{align*}
J_{n}^{(3)}&J_{m}^{(3)}+J_{n+1}^{(3)}J_{m+1}^{(3)}+J_{n+2}^{(3)}J_{m+2}^{(3)}\\
&=\frac{1}{49}\left(\begin{array}{c}\left(2^{n+1}-V_{n}\right)\left(2^{m+1}-V_{m}\right)+\left(2^{n+2}-V_{n+1}\right)\left(2^{m+2}-V_{m+1}\right)\\
+\left(2^{n+3}-V_{n+2}\right)\left(2^{m+3}-V_{m+2}\right)
\end{array}\right)\\
&=\frac{1}{49}\left(\begin{array}{c}21\cdot 2^{n+m+2}-2^{n+1}(V_{m}+2V_{m+1}+4V_{m+2})\\
-2^{m+1}(V_{n}+2V_{n+1}+4V_{n+2})+V_{n}V_{m}+V_{n+1}V_{m+1}+V_{n+2}V_{m+2}
\end{array}\right)\\
&=\frac{1}{7}\left(3\cdot2^{n+m+2}-2^{n+1}U_{m}-2^{m+1}U_{n}+W_{n-m}\right),
\end{align*}
where $V_{n+1}+3V_{n+2}=7U_{n}$ and $W_{n}=H_{n}(2,-1)=\omega_{1}^{n}+\omega_{2}^{n}$. Then,
\begin{align*}
\left\langle \overrightarrow{JD_{n}^{(3)}},\overrightarrow{JD_{m}^{(3)}} \right\rangle&=\frac{1}{7}\left(3\cdot2^{n+m+2}-2^{n+1}U_{m}-2^{m+1}U_{n}+W_{n-m}\right)\\
&\ \ +\frac{\varepsilon}{7}\left(3\cdot2^{n+m+4}+2^{n+1}W_{m+1}+2^{m+1}W_{n+1}-W_{n-m}\right)\\
&=\frac{1}{7} \left(\begin{array}{c} 
3\cdot2^{n+m+2}(1+4\varepsilon)-2^{n+1}(U_{m}-\varepsilon W_{m+1})\\
-2^{m+1}(U_{n}-\varepsilon W_{n+1})+W_{n-m}(1-\varepsilon)
\end{array}\right),
\end{align*}
with $U_{n+1}+2U_{n}=-W_{n+1}$, $W_{n}+W_{n+2}=-W_{n+1}$ and $UD_{n}=U_{n}+\varepsilon U_{n+1}$, we easily reach the result (\ref{p14}).
\end{proof}

\begin{theorem}\label{th:6}
For $n, m\geq 0$. Let $\overrightarrow{JD_{n}^{(3)}}$ and $\overrightarrow{JD_{m}^{(3)}}$ be two dual third order Jacobsthal vectors. The cross
product of $\overrightarrow{JD_{n}^{(3)}}$ and $\overrightarrow{JD_{m}^{(3)}}$ is given by
\begin{equation}\label{p15}
\overrightarrow{JD_{n}^{(3)}} \times \overrightarrow{JD_{m}^{(3)}}=\frac{1}{7} \left(\begin{array}{c} 
2^{n+1}(ZD_{m+1}+2\varepsilon Z_{m+1})-2^{m+1}(ZD_{n+1}+2\varepsilon Z_{n+1})\\
+U_{n-m}(1-\varepsilon)(\textbf{i}+\textbf{j}+\textbf{k})
\end{array}\right),
\end{equation}
where $Z_{n}=2U_{n+1}\textbf{i}+W_{n+1}\textbf{j}+U_{n}\textbf{k}$, $U_{n}=H_{n}(0,1)$, $W_{n}=H_{n}(2,-1)$, $\textbf{i}=(1,0,0)$, $\textbf{j}=(0,1,0)$, $\textbf{k}=(0,0,1)$ and $ZD_{n}=Z_{n}+\varepsilon Z_{n+1}$.
\end{theorem}
\begin{proof}
From the equations (\ref{p9}) and (\ref{ec:Gu}), we get 
\begin{equation}\label{p16}
\overrightarrow{JD_{n}^{(3)}} \times \overrightarrow{JD_{m}^{(3)}}=\overrightarrow{J_{n}^{(3)}}\times \overrightarrow{J_{m}^{(3)}} +\varepsilon \left(\overrightarrow{J_{n}^{(3)}}\times \overrightarrow{J_{m+1}^{(3)}} +\overrightarrow{J_{n+1}^{(3)}}\times \overrightarrow{J_{m}^{(3)}} \right).
\end{equation}
First, let us compute $\overrightarrow{J_{n}^{(3)}}\times \overrightarrow{J_{m}^{(3)}}$, if we use the properties of determinant to calculate the cross product of two vectors, the equality $$J_{n}^{(3)}J_{m+1}^{(3)}-J_{n+1}^{(3)}J_{m}^{(3)}=\frac{1}{7}\left( 
\begin{array}{ccc} 2^{n+1}U_{m+1}-2^{m+1}U_{n+1}+U_{n-m}
\end{array}%
\right)$$ (see \ref{p4}), $U_{n}=H_{n}(0,1)$ and simplify the statements, we find that
\begin{equation}\label{d1}
\begin{aligned}
\overrightarrow{J_{n}^{(3)}}\times \overrightarrow{J_{m}^{(3)}}&=\left\vert
\begin{array}{ccc}
\textbf{i}&\textbf{j}& \textbf{k}\\ 
J_{n}^{(3)}&J_{n+1}^{(3)}& J_{n+2}^{(3)} \\ 
J_{m}^{(3)}&J_{m+1}^{(3)}& J_{m+2}^{(3)}%
\end{array}%
\right\vert\\
&=\textbf{i} \left\vert
\begin{array}{cc}
J_{n+1}^{(3)}& J_{n+2}^{(3)} \\ 
J_{m+1}^{(3)}& J_{m+2}^{(3)}%
\end{array}%
\right\vert-\textbf{j} \left\vert
\begin{array}{cc}
J_{n}^{(3)}& J_{n+2}^{(3)} \\ 
J_{m}^{(3)}& J_{m+2}^{(3)}%
\end{array}%
\right\vert+\textbf{k} \left\vert
\begin{array}{ccc} 
J_{n}^{(3)}&J_{n+1}^{(3)}\\ 
J_{m}^{(3)}&J_{m+1}^{(3)}%
\end{array}%
\right\vert\\
&=\frac{1}{7} \left(\begin{array}{c} 
\textbf{i}(2^{n+2}U_{m+2}-2^{m+2}U_{n+2}+U_{n-m})\\
-\textbf{j}(-2^{n+1}W_{m+2}+2^{m+1}W_{n+2}-U_{n-m})\\
+\textbf{k}(2^{n+1}U_{m+1}-2^{m+1}U_{n+1}+U_{n-m})
\end{array}\right)\\
&=\frac{1}{7} \left(\begin{array}{c} 
2^{n+1}Z_{m+1}-2^{m+1}Z_{n+1}+U_{n-m}(\textbf{i}+\textbf{j}+\textbf{k})
\end{array}\right)
\end{aligned}
\end{equation}
where $Z_{n}=2U_{n+1}\textbf{i}+W_{n+1}\textbf{j}+U_{n}\textbf{k}$, $U_{n}=H_{n}(0,1)$, $W_{n}=H_{n}(2,-1)$, $\textbf{i}=(1,0,0)$, $\textbf{j}=(0,1,0)$ and $\textbf{k}=(0,0,1)$. Putting the equation (\ref{d1}) in (\ref{p16}), and using the definition of third order Jacobsthal numbers, we obtain the result as
\begin{align*}
\overrightarrow{JD_{n}^{(3)}} \times \overrightarrow{JD_{m}^{(3)}}&=\frac{1}{7} \left(\begin{array}{c} 
2^{n+1}Z_{m+1}-2^{m+1}Z_{n+1}+U_{n-m}(\textbf{i}+\textbf{j}+\textbf{k})
\end{array}\right)\\
&+\frac{\varepsilon}{7} \left(\begin{array}{c} 
2^{n+1}Z_{m+2}-2^{m+2}Z_{n+1}+U_{n-m-1}(\textbf{i}+\textbf{j}+\textbf{k})\\
+2^{n+2}Z_{m+1}-2^{m+1}Z_{n+2}+U_{n-m+1}(\textbf{i}+\textbf{j}+\textbf{k})
\end{array}\right)\\
&=\frac{1}{7} \left(\begin{array}{c} 
2^{n+1}(ZD_{m+1}+2\varepsilon Z_{m+1})-2^{m+1}(ZD_{n+1}+2\varepsilon Z_{n+1})\\
+U_{n-m}(1-\varepsilon)(\textbf{i}+\textbf{j}+\textbf{k})
\end{array}\right),
\end{align*}
where $ZD_{m}=Z_{m}+\varepsilon Z_{m+1}$.
\end{proof}

\end{document}